\documentclass[10pt]{amsart}

\usepackage{graphicx}
\usepackage{amsfonts,amsmath,latexsym,amssymb,amsthm}
\usepackage{verbatim}

\newtheorem{theorem}{Theorem}[section]
\newtheorem{lemma}[theorem]{Lemma}
\newtheorem{proposition}[theorem]{Proposition}
\newtheorem{corollary}[theorem]{Corollary}

\theoremstyle{definition}

\newtheorem{question}[theorem]{Question}

\theoremstyle{remark}

\numberwithin{equation}{section}

\newcommand{\IQ}{\mathbb{Q}}
\newcommand{\Qbar}{\overline{\IQ}}

\newcommand{\coll}{{\mathcal{C}}}
\newcommand{\lb}{\eta}
\newcommand{\ub}{\theta}
\newcommand{\ratio}{\gamma}

\begin{document}

\title{Number fields without small generators}

\author{Jeffrey D. Vaaler}
\address{Department of Mathematics, University of Texas at Austin, 1 University Station C1200, Austin, Texas 78712}
\email{vaaler@math.utexas.edu}

\author{Martin Widmer}
\address{Department of Mathematics\\ 
Royal Holloway, University of London\\ 
TW20 0EX Egham\\ 
UK}
\email{martin.widmer@rhul.ac.uk}

\date{\today, and in revised form ....}

\dedicatory{}


\begin{abstract}
Let $D>1$ be an integer, and let $b=b(D)>1$ be its smallest divisor.
We show that there are infinitely many number fields of degree $D$
whose primitive elements all have relatively large height in terms of $b$, $D$ and the
discriminant of the number field. This provides a negative answer 
to a questions of W. Ruppert from 1998 in the case when $D$ is composite.
Conditional on a very weak form of a folk conjecture about the distribution of  
number fields, we negatively answer Ruppert's question for all $D>3$. 
\end{abstract}

\maketitle



\section{Introduction}\label{introductionchapter3}
Let $H(\cdot)$ be the absolute multiplicative Weil height on the algebraic numbers $\Qbar$, as defined in, e.g., \cite{BG}, and 
let $L\subset \Qbar$ be a number field of degree $D>1$.   
We are interested in bounds, expressed  in terms of the degree $D$ and the absolute discriminant $\Delta_L$ of $L$, 
for the smallest height of a generator. 
It is convenient to use the following invariant, introduced by Roy and Thunder \cite{8},
\begin{alignat*}1
\delta(L)=\inf\{H(\alpha);L=\IQ(\alpha)\}.
\end{alignat*}
By Northcott's Theorem the infimum is attained, and hence, $\delta(L)$ denotes the smallest height
of a generator of the extension $L$ over $\IQ$.
Silverman \cite[Theorem 1]{9} has shown that
\begin{alignat}1
\label{sil2}
\delta(L)\geq D^{-\frac{1}{2(D-1)}}
|\Delta_{L}|^{\frac{1}{2D(D-1)}}.
\end{alignat}
The following example due to Ruppert \cite[p.18]{Rupp} and Masser \cite[Proposition 1]{8}) 
shows that in this general situation the exponent $1/(2D(D-1))$ cannot be improved.
Let $p$ and $q$ be primes that satisfy $0<p<q<2p$.
Let $\alpha=(p/q)^{1/D}$, and let $L=\IQ(\alpha)$. Then, by the Eisenstein criterion, $L$ has degree $D$, 
and $p$ and $q$ are both totally ramified in $L$. Hence,  $(pq)^{D-1}|\Delta_L$, and thus
\begin{alignat}1\label{example}
H(\alpha)=q^{\frac{1}{D}}\leq (2pq)^{{\frac{1}{2D}}}
\leq2^{\frac{1}{2D}}|\Delta_{L}|^{\frac{1}{2D(D-1)}}.
\end{alignat}
Ruppert \cite[Question 1]{Rupp} asked whether the exponent is always sharp,
more precisely he proposed the following question.
\begin{question}[Ruppert, 1998]\label{Ruppert1}
Is there a constant $C_D$ such that for all number fields $L$ of degree $D$ 
\begin{alignat*}1
\delta(L)\leq C_D
|\Delta_{L}|^{\frac{1}{2D(D-1)}}?
\end{alignat*}
\end{question}
In fact Ruppert used the naive height but elementary inequalities between 
the naive and the Weil height show that Ruppert's  question is equivalent to the one stated above. 
Ruppert \cite[Proposition 2]{Rupp} himself answered this question in the affirmative for $D=2$.
The aim of this note is to answer Ruppert's question in the negative for all 
composite $D$.
\begin{theorem}\label{th1}
Let $b=b(D)>1$ be the smallest divisor of $D$, 
and suppose $\ratio$ is a real number such that
\begin{alignat*}1
\ratio<\left\{
    \begin{array}{lll}
      &1/(D(b+1)):&\text{ if }b\leq 3,\\
      &{1}/{(2D(b+1))}+{1}/{(Db^2(b+1))}:&\text{ otherwise.}
\end{array} \right.
\end{alignat*}
Then there exist infinitely many number fields $L$ of degree $D$ satisfying 
\begin{alignat*}1
\delta(L)> |\Delta_L|^{\ratio}.
\end{alignat*}
\end{theorem}
Note that for composite $D>4$ we have
$$\frac{1}{2D(b+1)}+\frac{1}{Db^2(b+1)}>\frac{1}{2D(\sqrt{D}+1)}>{\frac{1}{2D(D-1)}}.$$
Thus, Theorem \ref{th1} provides a negative answer to Question \ref{Ruppert1} for all composite $D$. 
In fact, we prove a stronger result, namely: 
let $F$ be any number field  of degree $D/b$;
when enumerated by the modulus of the discriminant, the subset of all degree $b$ extensions 
$L$ of $F$, defined by $\delta(L)>|\Delta_L|^{\ratio}$, has density $1$ (for the precise statement we refer to 
Corollary \ref{cor1}).

Our proof strategy requires a good lower bound for the number of degree $D$ fields with bounded 
modulus of the discriminant. Essentially optimal bounds are available when $D$ is even or divisible by 
$3$, and if $D$ is composite we still have some useful bounds. However, a folk conjecture (sometimes attributed to Linnik)
predicts the asymptotics $c_D T$ as $T$ goes to infinity, for some constant $c_D>0$.  Unfortunately, the best general lower
bounds are only of order $T^{1/2+1/D^2}$ which is just slightly weaker than what we need. Therefore,  our next result
is conditional.
 
\begin{theorem}\label{th2}
Suppose that  $D>3$, and suppose that there exist constants $c_D>0$ and $\nu_D>1/2+1/(D-1)$ such that 
the number of degree $D$ fields $L\subset \Qbar$ with absolute value of the discriminant no larger than $T$
is at least $c_DT^{\nu_D}$ for all $T$ large enough.
Then there exists $\ratio>1/(2D(D-1))$ such that there are infinitely many number fields $L$ of degree $D$ with
\begin{alignat*}1
\delta(L)> |\Delta_L|^{\ratio}.
\end{alignat*}
\end{theorem}
Thanks to \cite{Bhargava10}, the hypothesis of 
Theorem \ref{th2} is satisfied for $D=5$, and hence we get an unconditional negative answer to Question \ref{Ruppert1}
for $D=5$. Furthermore, Theorem \ref{th2} shows that most likely the answer to Question \ref{Ruppert1} is ``no" for all $D>3$.
However, our method  sheds no light on the case $D=3$. 

In this article we use Vinogradov's notation $\ll$ and $\gg$ at various places.
The involved constants are allowed to depend on all quantities except on the parameter $T$, which is introduced in the next section.

\section{Enumerating fields: discriminant versus delta-invariant}\label{enumerating}

Any number field is considered a subfield of the fixed algebraic closure $\Qbar$.
Let $k$ be a number field, let $m=[k:\IQ]$,  let $L/k$ be a finite extension of degree $d=[L:k]>1$, and 
put $D=[L:\IQ]=md$.
For the remainder of this paper we set
\begin{alignat}1\label{coll}
\coll=\coll_{d}(k)=\{L\subset \Qbar; [L:k]=d\},
\end{alignat}
and for a subset $S\subset\coll$, and $\ratio\geq 0$ we set
\begin{alignat}1\label{Seta}
S_{\ratio}=\{L\in S; \delta(L)>|\Delta_{L}|^{\ratio}\}.
\end{alignat}
We want to enumerate the fields in $S$ in two different ways: once by the discriminant (more precisely, the modulus 
thereof), and once by the delta invariant $\delta(\cdot)$. Thus we introduce the counting functions
\begin{alignat*}1
N_{\Delta}(S,T)&=|\{L\in S; |\Delta_{L}|\leq T\}|,\\
N_{\delta}(S,T)&=|\{L\in S; \delta(L)\leq T\}|.
\end{alignat*}
Note that both cardinalities are finite; the first one by Hermite's Theorem, the second one by Northcott's Theorem.
Next we introduce the set of generators of fields of $S$
\begin{alignat*}1
P_S=\{\alpha\in \Qbar; \IQ(\alpha)\in S\},
\end{alignat*}
and its counting function
\begin{alignat*}1
N_H(P_S,T)=|\{\alpha\in P_S; H(\alpha)\leq T\}|.
\end{alignat*}
Again, the above cardinality is finite by Northcott's Theorem.

The proof of Theorem \ref{th1} is based on two simple observations, the first of which, is presented as the following proposition. 
\begin{proposition}\label{propVaalerW1}
Suppose there are positive reals $\lb$, $\ub$, and $\ratio<\lb/\ub$ such that $N_{\Delta}(S,T)\gg T^\lb$ and $N_H(P_S,T)\ll T^\ub$ for all $T$ large enough. Then
\begin{alignat*}1
\lim_{T\rightarrow \infty}\frac{N_{\Delta}(S_{\ratio},T)}{N_{\Delta}(S,T)}=1.
\end{alignat*}
\end{proposition}
\begin{proof}
Directly from the definitions we get
\begin{alignat*}1
N_{\Delta}(S\backslash S_{\ratio},T)\leq N_{\delta}(S\backslash S_{\ratio},T^{\ratio})\leq N_{\delta}(S,T^{\ratio}).
\end{alignat*}
The map $\alpha\rightarrow \IQ(\alpha)$ yields a surjection from $\{\alpha\in P_S; H(\alpha)\leq T^\ratio\}$ to
$\{L\in S; \delta(L)\leq T^\ratio\}$.
Hence, we have
\begin{alignat*}1
N_{\delta}(S,T^{\ratio})\leq N_H(P_S,T^{\ratio}).
\end{alignat*}
On the other hand, by the hypothesis,
\begin{alignat*}1
N_H(P_S,T^{\ratio})\ll T^{\ratio\ub},
\end{alignat*}
and
\begin{alignat*}1
N_{\Delta}(S,T)\gg T^\lb,
\end{alignat*}
provided $T$ is large enough. 
We conclude
\begin{alignat*}1
\lim_{T\rightarrow \infty}\frac{N_{\Delta}(S\backslash S_{\ratio},T)}{N_{\Delta}(S,T)}=0,
\end{alignat*}
whenever $\ratio<\frac{\lb}{\ub}$ which proves the proposition.
\end{proof}

\section{Bounds for the counting functions}\label{counting}
In view of Proposition \ref{propVaalerW1} we want to find a set $S\subset \coll$ that maximizes the ratio $\lb/\ub$.
Taking $S=\coll_b(F)\subset \coll$ as the set of fields that contain a fixed extension $F/k$ of degree  $d/b$ does not 
affect $\lb$ in a negative way as we shall see in Lemma \ref{lowerbound}, but it positively affects $\ub$ as we shall see
in Lemma \ref{upperbound}. This is the  second simple but 
important observation for the proof of Theorem \ref{th1}. 

We start with lower bounds for $\lb$, that is,  lower bounds for $N_{\Delta}(\coll_{b}(F),T)$.
\begin{lemma}\label{lowerbound}
Let  $b=b(d)>1$ be the smallest divisor of $d$, and let $F$ be an extension of $k$ of degree $d/b$.
Then we have
\begin{alignat}1\label{EllenbergVenkatesh}
N_{\Delta}(\coll_{b}(F),T)\gg T^{1/2+1/b^2}
\end{alignat}
for all $T$ large enough.
If $d$  is even or divisible by $3$ then we even have
\begin{alignat}1\label{quadrext}
N_{\Delta}(\coll_{b}(F),T)\gg T, 
\end{alignat}
for all $T$ large enough.
\end{lemma}
\begin{proof}
First we recall that for $L\in \coll_b(F)$ we have $|\Delta_L|=|\Delta_F|^bN_{F/\IQ}(\mathfrak{D}_{L/F})$,
where $N_{F/\IQ}(\cdot)$ is the absolute norm, and $\mathfrak{D}_{L/F}$ is the relative discriminant. Thus, counting fields in $\coll_b(F)$
with $|\Delta_L|\leq T$ is the same as counting fields in $\coll_b(F)$
with $N_{F/\IQ}(\mathfrak{D}_{L/F})\leq T/|\Delta_F|^b$. Therefore,
Ellenberg and Venkatesh's \cite[Theorem 1.1]{56} shows that
\begin{alignat*}1
N_{\Delta}(\coll_b(F),T)\geq c'T^{1/2+1/b^2} 
\end{alignat*}
for some $c'=c'(b,F)>0$ and all $T$ large enough. This yields (\ref{EllenbergVenkatesh}).
For (\ref{quadrext}) we note that the conjectured asymptotic formula
\begin{alignat*}1
N_{\Delta}(\coll_b(F),T)=cT+o(T), 
\end{alignat*}
where $c=c(b,F)>0$, has been proven by  Datskovsky and Wright for $b=2$  \cite[Theorem 4.2]{82} (see also \cite[Corollary 1.2]{59})
and for $b=3$  \cite[Theorem 1.1]{82}. This proves the lemma. 
\end{proof}

Next we establish an upper bound for $N_{H}(P_{\coll_{b}(F)},T)$.
Recall that $k$ is a number field of degree $m$, and also recall the notation $\coll=\coll_d(k)$ from
(\ref{coll}).
\begin{lemma}\label{upperbound}
We have for all $T>0$
\begin{alignat}1
\label{Schmidt}
N_H(P_\coll,T)\ll T^{ md(d+1)}.
\end{alignat}
With the notation of Lemma \ref{lowerbound}, in particular, 
\begin{alignat}1
\label{SchmidtFb}
N_H(P_{\coll_b(F)},T) \ll T^{ md(b+1)}.
\end{alignat}
\end{lemma}
\begin{proof}
First we note that $\IQ(\alpha)\in \coll=\coll_d(k)$ implies $[k(\alpha):k]=d$. Therefore, 
\begin{alignat*}1
N_H(P_\coll,T)\leq |\{\alpha\in \Qbar; [k(\alpha):k]=d, H(\alpha)\leq T\}|. 
\end{alignat*}
Now Schmidt \cite[Theorem]{22}  has shown that
\begin{alignat*}1
|\{\alpha\in \Qbar; [k(\alpha):k]=d, H(\alpha)\leq T\}|\leq C(m,d)T^{md(d+1)}. 
\end{alignat*}
Therefore
$N_H(P_\coll,T)\leq C(m,d)T^{md(d+1)}$, which proves (\ref{Schmidt}).
\end{proof}

\section{Density results}\label{densityresults}
Recall  the notation in (\ref{Seta}).
\begin{corollary}\label{cor1}
Let $b=b(d)>1$ be the smallest divisor of $d$, 
and suppose $\ratio$ is a real number such that
\begin{alignat*}1
\ratio<\left\{
    \begin{array}{lll}
      &1/(md(b+1)):&\text{ if }b\leq 3,\\
      &{1}/{(2md(b+1))}+{1}/{(mdb^2(b+1))}:&\text{ otherwise.}
\end{array} \right.
\end{alignat*}
Let $F$ be an extension of $k$ of degree $d/b$ and let $B=\{L\in \coll; F\subset L\}$.
Then 
\begin{alignat*}1
\lim_{T\rightarrow \infty}\frac{N_{\Delta}(B_{\ratio},T)}{N_{\Delta}(B,T)}=1.
\end{alignat*}
\end{corollary}
\begin{proof}
First note that $B=\coll_{b}(F)$.
Thus  (\ref{EllenbergVenkatesh}) yields $N_{\Delta}(B,T)\gg T^{1/2+1/b^2}$ for $T$ large enough.
If  $d$ is even or divisible by  $3$ then by (\ref{quadrext}) we even have
$N_{\Delta}(B,T)\gg T$ for $T$ large enough.
On the other hand  (\ref{SchmidtFb}) gives 
$N_H(P_B,T)\ll T^{md(b+1)}$. 
Applying Proposition \ref{propVaalerW1} with $S=B$ yields the statement.
\end{proof}
So almost all fields in $B=\coll_{b}(F)$ satisfy $\delta(L)> |\Delta|^\ratio$.
Note that, of course, $B$ is an infinite set, and so Theorem \ref{th1} follows from Corollary \ref{cor1} by taking $k=\IQ$.
\begin{corollary}\label{rmklb}
Suppose $\ratio<1/(md(d+1))$ and suppose $d$ is even or divisible by $3$ then 
\begin{alignat*}1
\lim_{T\rightarrow \infty}\frac{N_{\Delta}(\coll_{\ratio},T)}{N_{\Delta}(\coll,T)}=1.
\end{alignat*}
\end{corollary}
\begin{proof}
Let $F$ be an extension of $k$ of degree $d/2$ if $d$ is even, and  of degree $d/3$ otherwise. Hence, $\coll\supset \coll_2(F)$ or $\coll\supset \coll_3(F)$ respectively, and so we conclude from (\ref{quadrext})
that  $N_{\Delta}(\coll,T)\gg T$.  Furthermore, by (\ref{Schmidt}) we have
$N_H(P_\coll,T)\ll T^{md(d+1)}$. Applying Proposition \ref{propVaalerW1} with $S=\coll$ yields the statement.
\end{proof}

Finally, to prove Theorem \ref{th2} we apply Proposition \ref{propVaalerW1}
with $S=\coll$, $k=\IQ$, $\lb=\nu_D>1/2+1/(D-1)$ and $\ub=D(D+1)$ (for the latter we have applied (\ref{Schmidt})). As $\lb/\ub>1/(2D(D-1))$ we conclude that there
exists $\ratio>1/(2D(D-1))$ such that there exist infinitely many number fields $L$ of degree $D$ that satisfy
\begin{alignat*}1
\delta(L)> |\Delta_{L}|^{\ratio}.
\end{alignat*}
This proves Theorem \ref{th2}.

\section{Cluster points}
We consider the set of values
\begin{alignat*}1
\frac{\log \delta(L)}{\log|\Delta_{L}|}
\end{alignat*}
as $L$ runs over all number fields of fixed degree $D>1$. What are the cluster points of this set? 
Combining (\ref{sil2}) and (\ref{example}) gives the smallest
cluster point
\begin{alignat*}1
\lim \inf_{[L:\IQ]=D}\frac{\log\delta(L)}{\log|\Delta_L|}=\frac{1}{2D(D-1)}.
\end{alignat*}
What about the largest cluster point?
With $b=b(D)$ as in Theorem \ref{th1} the latter implies that
\begin{alignat*}1
\lim \sup_{[L:\IQ]=D}\frac{\log\delta(L)}{\log|\Delta_L|}\geq \left\{
    \begin{array}{lll}
      &1/(D(b+1)):&\text{ if }b\leq 3,\\
      &{1}/{(2D(b+1))}+{1}/{(b^2(b+1)D)}:&\text{ otherwise.}
\end{array} \right.
\end{alignat*}
If $D$ is odd \cite[Theorem 1.2]{VaalerWidmer} or if the Dedekind zeta-function associated to the 
Galois closure of $L$ satisfies the Generalized Riemann Hypothesis
for all number fields $L$ of degree $D$, then \cite[Theorem 1.3]{VaalerWidmer}
\begin{alignat*}1
\lim \sup_{[L:\IQ]=D}\frac{\log\delta(L)}{\log|\Delta_L|}\leq 1/(2D).
\end{alignat*}
However, the best known unconditional general upper bound for the largest cluster point is $1/D$, see, e.g., \cite[Lemma 4.5]{art2}.
It might be an interesting problem to study the distribution of the cluster points, and to locate new cluster points.

\section*{Acknowledgments}
Parts of this article were discussed and written during the special semester ``Heights in Diophantine Geometry, Group Theory and Additive Combinatorics'' held at the Erwin Schr\"odinger International Institute for Mathematical Physics.

\bibliographystyle{amsplain}
\bibliography{literature}

\end{document}